\documentclass[11pt,reqno]{amsart}

\usepackage{amsthm,amssymb}
\usepackage[pagebackref,colorlinks,linkcolor=red,citecolor=blue,urlcolor=blue,hypertexnames=false]{hyperref}
\usepackage{amsrefs}
\usepackage{amscd}
\usepackage{undertilde}
\usepackage{mathrsfs}
\input xypic
\xyoption{all}
\newdir{ >}{{}*!/-5pt/@{>}}

\newcommand{\Cliff}{\mathrm{Cliff}}
\newcommand{\topdf}{\texorpdfstring}

\newcommand\ol{\overline}
\newcommand\ul{\underline}

\newcommand\Da{\overset{\circ}{D}}
\def\cA{\mathcal A}
\def\cAc{\cA_c}

\def\cC{\mathcal C}
\def\cB{\mathcal B}
\renewcommand{\cH}{\mathcal H}

\def\cZ{\mathcal Z}
\def\cZc{\cZ_c}

\def\cE{\mathcal E}

\def\fX{\mathfrak X}

\def\E{\mathbb{E}}

\def\fQ{\mathfrak{Q}_G}

\def\cF{\mathcal F}
\def\fB{\mathfrak B}
\def\fC{\mathfrak C}
\def\fD{\mathfrak D}
\def\fH{\mathfrak H}

\def\sU{\mathscr U}
\def\fall{\mathcal All}
\def\cG{\mathcal G}

\def\cK{\mathcal K}

\def\alg{\mathrm{Alg}}
\def\ring{\mathrm{Rings}}

\def\cat{\mathrm{Cat}}

\def\spt{\mathrm{Spt}}

\def\Bor{\mathfrak {BC}^*}

\def\ab{\mathfrak {Ab}}

\def\supp{\mathrm{supp}}
\def\red{\mathrm{red}}
\renewcommand\top{\mathrm{top}}

\def\even{\mathrm{even}}
\def\odd{\mathrm{odd}}

\def\ev{\mathrm{ev}}

\def\sotimes{\utilde{\otimes}}
\def\motimes{\otimes_{\mu}}
\def\gotimes{\hat{\sotimes}}

\newcommand{\C}{\mathbb{C}}

\newcommand{\Q}{\mathbb{Q}}
\newcommand{\R}{\mathbb{R}}
\newcommand{\F}{\mathbb{F}}
\newcommand{\G}{\mathbb{G}}
\newcommand{\fin}{\mathcal{F}in}
\newcommand{\vcyc}{\mathcal{V}cyc}

\newcommand{\Z}{\mathbb{Z}}
\newcommand{\N}{\mathbb{N}}

\newcommand{\ho}{\mathrm{Ho}}
\newcommand{\hofi}{\mathrm{hofiber}}

\newcommand{\fA}{\mathfrak{A}}

\newcommand{\org}{\mathrm{Or}G}

\def\bu{\bullet}

\def\colim{\operatornamewithlimits{colim}}

\def\lra{\longrightarrow}
\def\onto{\twoheadrightarrow}

\def\triqui{\vartriangleleft}
\def\weq{\overset\sim\lra}
\def\lweq{\overset\sim\longleftarrow}
\def\fibeq{\overset\sim\onto}

\numberwithin{equation}{section}
\theoremstyle{plain}
\newtheorem{thm}[equation]{Theorem}
\newtheorem{cor}[equation]{Corollary}
\newtheorem{lem}[equation]{Lemma}
\newtheorem{prop}[equation]{Proposition}

\newcommand{\comment}[1]{}  

\theoremstyle{definition}

\theoremstyle{remark}
\newtheorem{rem}[equation]{Remark}
\newtheorem{ex}[equation]{Example}

\begin{document}
\title[Algebraic $K$-theory for groups acting properly on Hilbert space]{Compact operators and algebraic $K$-theory for groups which act properly and isometrically on Hilbert space}
\author{Guillermo Corti\~nas}
\email{gcorti@dm.uba.ar}\urladdr{http://mate.dm.uba.ar/\~{}gcorti}
\author{Gisela Tartaglia}
\email{gtartag@dm.uba.ar}
\address{Dep. Matem\'atica-IMAS, FCEyN-UBA\\ Ciudad Universitaria Pab 1\\
1428 Buenos Aires\\ Argentina}
\thanks{The first author was partially supported by MTM2012-36917-C03-02. Both authors
were supported by CONICET, and partially supported by
grants UBACyT 20020100100386 and PIP 11220110100800.}
\begin{abstract}
We prove the $K$-theoretic Farrell-Jones conjecture for\linebreak
groups as in the title with
coefficient rings and $C^*$-algebras which are stable with respect to compact operators. We use this and Higson-Kasparov's result that the Baum-Connes conjecture with coefficients holds for such groups, to show that if $G$ is as in the title then the algebraic and the $C^*$-crossed products of $G$ with a stable separable $G$-$C^*$-algebra have the same
$K$-theory. 
\end{abstract}
\maketitle

\section{Introduction}\label{sec:intro}

Let $G$ be a group; a \emph{family} of subgroups of $G$ is a
nonempty family $\cF$ closed under conjugation and under taking
subgroups. A \emph{$G$-space} is a simplicial set together with a $G$-action. 
If $\cF$ is a family of subgroups of $G$ and $f:X\to Y$ is an equivariant
map of $G$-spaces, then we say that $f$ is an \emph{$\cF$-equivalence} (resp. an \emph{$\cF$-fibration}) if the map between fixed
point sets
\[
f:X^H\to Y^H
\]
is a weak equivalence (resp. a fibration) for every $H\in \cF$. A
$G$-space $X$ is called a {\em $(G,\cF)$-complex} if the
stabilizer of every simplex of $X$ is in $\cF$. The category of
$G$-spaces can be equipped with a closed model structure
where the weak equivalences (resp. the fibrations) are the $\cF$-equivalences (resp. the $\cF$-fibrations), 
(see \cite{corel}*{\S1}). The $(G,\cF)$-complexes
are the cofibrant objects in this model structure. By a general
construction of Davis and L\"uck (see \cite{dl}) any functor $E$
from the category $\Z$-$\cat$ of small $\Z$-linear categories to the
category $\spt$ of spectra which sends category equivalences to
weak equivalences of spectra gives rise to an equivariant homology theory
of $G$-spaces $X\mapsto H^G(X,E(R))$ for each unital $G$-ring $R$. If $H\subset G$ is a subgroup, 
then
\begin{equation}\label{intro:cross}
H_*^G(G/H,E(R))=E_*(R\rtimes H)
\end{equation}
is just $E_*$ evaluated at the crossed product ring. The \emph{isomorphism
conjecture} for the quadruple $(G,\cF,E,R)$ asserts that if
$\cE(G,\cF)\fibeq pt$ is a $(G,\cF)$-cofibrant replacement of the
point, then the induced map
\begin{equation}\label{intro:assem}
H_*^G(\cE(G,\cF),E(R))\to E_*(R\rtimes G)
\end{equation}
--called \emph{assembly map}-- is an isomorphism. For the family
$\cF=\fall$ of all subgroups, \eqref{intro:assem} is always an
isomorphism. The appropriate choice of $\cF$ varies with $E$.  For
$E=K$, the nonconnective algebraic $K$-theory spectrum, one takes
$\cF=\vcyc$, the family of virtually cyclic subgroups; the isomorphism conjecture
for $(G,\vcyc,K,R)$ is the $K$-theoretic
\emph{Farrell-Jones conjecture} with coefficients in $R$. Moreover, for $E=K$, \eqref{intro:assem} makes
sense for those coefficient rings $R$ which are \emph{$K$-excisive},
i.e. those for which $K$-theory satisfies excision (\cite{corel}*{Section ~3}). In this paper we are interested in the $K$-theory isomorphism conjecture for coefficient rings of the form
\begin{equation}\label{anillo}
R=I\otimes (\fA\sotimes \cK)
\end{equation}
where $I$ is a $K$-excisive $G$-ring, $\fA$ is a complex $G$-$C^*$-algebra (or more generally a $G$-bornolocal $C^*$-algebra as defined in Section \ref{sec:bornoloco}), $\otimes=\otimes_\Z$ is the algebraic tensor product, $\sotimes$ is the spatial tensor product, and $\cK$ is the ideal of compact operators in an infinite dimensional, separable,
complex Hilbert space with trivial $G$-action. 
We consider the Farrell-Jones conjecture for discrete groups having the \emph{Haagerup approximation property}. These are the countable discrete groups which admit an
affine, isometric and \emph{metrically proper} action on a real pre-Hilbert space $V$ of countably infinite dimension (or equivalently on a Hilbert space). The term
metrically proper means that for every $v\in V$,
\[
\lim_{g\to\infty}||gv||=\infty.
\]
The groups satisfying this property are also called \emph{a-T-menable}, a term coined by Gromov (\cite{gromo}).
Our main result is the following (see Theorem \ref{thm:main}).

\begin{thm}\label{thm:intro}
Let $G$ be a countable discrete group. Let $\fA$ be a $G$-$C^*$-algebra, let $I\in G$-Ring, and let $\cK=\cK(\ell^2(\N))$ be the algebra of compact operators; equip $\cK$ with the trivial $G$-action. Assume that $I$ is $K$-excisive and that $G$ has the Haagerup approximation property. Let $\fin$ be the family of finite subgroups. Then the functor $H^G(-,K(I\otimes(\fA\sotimes\cK)))$ sends $\fin$-equivalences of $G$-spaces to weak equivalences of spectra.
\end{thm}

Observe that because $\vcyc\supset\fin$, any $\vcyc$-equivalence is also a $\fin$-equivalence. Since $\cE(G,\vcyc)\to pt$ is a $\vcyc$-equivalence by definition, the theorem has the following corollary (see Corollary \ref{cor:main1}).

\begin{cor}\label{cor:intro1}
Let $G$, $I$ and $\fA$ be as in Theorem \ref{thm:intro}. Then $G$ satisfies the $K$-theoretic Farrell-Jones conjecture with coefficients
in $I\otimes(\fA\sotimes\cK)$. 
\end{cor}

Higson and Kasparov proved in \cite{hk2} that the groups which have the Haagerup approximation property satisfy the \emph{Baum-Connes conjecture} with coefficients in any separable $G$-$C^*$-algebra. The latter conjecture is the analogue of the Farrell-Jones conjecture for the topological $K$-theory of reduced $C^*$-crossed products. It asserts that the assembly map
\[
H^G(\cE(G,\fin),K^{\top}(\fA))\to K^{\top}(C^*_\red(G,\fA))
\] 
is a weak equivalence. Here $C^*_\red(G,\fA)$ is the reduced $C^*$-algebra and $H^G(-,K^{\top}(\fA))$ is equivariant topological $K$-homology. The latter homology is characterized by
\[
H_*^G(G/H,K^{\top}(\fA))=K^{\top}_*(C^*_\red(H,\fA)).
\]
There is a natural map
\begin{equation}\label{map:cross}
\fA\rtimes H\to C^*_\red(H,\fA)
\end{equation}
which is an isomorphism when $H$ is finite. We have a homotopy commutative diagram
\begin{equation}\label{intro:diag}
\xymatrix{H^G(\cE(G,\fin),K(\fA))\ar[d]\ar[r]& K(\fA\rtimes G)\ar[d]\\
H^G(\cE(G,\fin),K^{\top}(\fA))\ar[r]& K^{\top}(C^*_\red(G,\fA)).
}
\end{equation}
It follows from Suslin-Wodzicki's theorem (Karoubi's conjecture) (\cite{sw}*{Theorem 10.9}) and the facts that \eqref{map:cross} is an isomorphism for finite $H$, and that $G$ acts on $\cE(G,\fin)$ with finite stabilizers, that the vertical map on the left of \eqref{intro:diag} is a weak equivalence whenever $\fA$ is of the form $\fA=\fB\sotimes\cK$. Using this, the stability of $K^{\top}$ under tensoring with $\cK$, and Higson-Kasparov's result, we obtain the following corollary of Theorem \ref{thm:intro} (see Corollary \ref{cor:main2}).

\begin{cor}\label{cor:intro2}
Let $G$ and $\fA$ be as in Theorem \ref{thm:intro}. Assume that $\fA$ is separable. Then there is an isomorphism:
\[
K_*((\fA\sotimes\cK)\rtimes G)\cong K^{\top}_*(C^*_\red(G,\fA)).
\]
\end{cor}

Higson and Kasparov showed in \cite{hk2}*{Theorem 9.4} that if $G$ is a locally compact group which has the Haagerup property, $\fA$ is a separable $G$-$C^*$-algebra, and $C^*(G,\fA)$ is the full crossed product, then the map
\[
K_*^{\top}(C^*(G,\fA))\to K_*^{\top}(C^*_\red(G,\fA))
\]
is an isomorphism. Hence in Corollary \ref{cor:intro2} we may substitute the full $C^*$-crossed product for the reduced one. 

\goodbreak

The rest of this paper is organized as follows. In Section \ref{sec:bornoloco} we give some preliminaries on bornolocal $C^*$-algebras. These are normed $*$-algebras over $\C$ such that $\left\|a^*a\right\|=\left\|a\right\|^2$, possibly not complete, which are filtered unions of $C^*$-subalgebras.
For example, if $X$ is a locally compact Hausdorff topological space and $\fA$ is a $C^*$-algebra, then the algebra $C_c(X,\fA)$ of compactly supported continuous functions $X\to \fA$ is a bornolocal $C^*$-algebra.
We write $\Bor$ for the category of bornolocal $C^*$-algebras and $G$-$\Bor$ for the corresponding equivariant category.
In Section \ref{sec:equihomo} we prove Theorem \ref{thm:kstab}, which says that if $G$ is a discrete group, $X$ a $G$-space, $I$ a $K$-excisive $G$-ring and $B\in G$-$\Bor$, then the functor 
\begin{equation}\label{fun:intro}
\E_X:G\emph{-}\Bor\to \spt, \ \ A\mapsto H^G(X, K(I\otimes A\motimes \cK\motimes B))
\end{equation}
is excisive, homotopy invariant, and $G$-stable. Here $\motimes$ is the maximal tensor product of bornolocal $C^*$-algebras; it is defined in \eqref{motimes}. Section \ref{sec:asy} concerns equivariant asymptotic homomorphisms of $G$-bornolocal $C^*$-algebras. In this technical section, we discuss how to extend the functor \eqref{fun:intro} to a functor $\bar{\E}_X$ that can be applied to certain equivariant asymptotic homomorphisms; the main results of this section are Proposition \ref{prop:asy} and Lemma \ref{lem:conmuta}. In Section \ref{sec:dirac} we recall Higson-Kasparov's construction of a dual Dirac element in equivariant $E$-theory (\cite{hk2}). For a group $G$ which acts by affine isometries on a countably infinite dimensional Euclidean space $V$, they construct a $G$-$C^*$-algebra $\cA_0(V)$ which is a $C^*$-colimit 
over all finite dimensional subspaces $S\subset V$, of algebras of continuous functions $\R\times S\to \Cliff(\R\oplus S)$, vanishing at infinity, and taking values in the complexified Clifford algebra $\Cliff(\R\oplus S)$. They define an equivariant asymptotic homomorphism
\begin{equation}\label{intro:beta0}
\hat{\beta_0}:C_0(\R)--> \cA_0(V),
\end{equation}
and they show that its class in $E_G(C_0(\R),\cA_0(V))$, which they call the \emph{dual Dirac element}, is invertible. We define a bornolocal $G$-$C^*$-algebra $\cAc(V)$ which is an algebraic colimit, over the finite dimensional subspaces $S\subset V$, of algebras of compactly supported continuous functions $\R\times S\to \Cliff(\R\oplus S)$. The map \eqref{intro:beta0} restricts to an equivariant asymptotic homomorphism
\[
\hat{\beta_c}:C_c(\R)-->\cAc(V).
\]
Using Higson-Kasparov's result, together with Lemma \ref{lem:conmuta}, we show in Proposition \ref{prop:bott} that for the extension $\bar{\E}_X$ of Section \ref{sec:asy}, $\bar{\E}_X(\hat{\beta_c})$ has a left homotopy inverse. Then in Corollary \ref{cor:bott} we give the following application. Let $f:X\to Y$ be an equivariant map and let
\begin{equation}\label{intro:tau}
\E_X\to \E_Y
\end{equation}
be the natural transformation induced by $f$.  Then $$\E_X(\C)\to \E_Y(\C)$$ is a weak equivalence whenever $\E_X(\cAc(V))\to \E_Y(\cAc(V))$ is.
In Section \ref{sec:proper} we recall the notion of proper $G$-rings over a discrete homogeneous space $G/H$, introduced in \cite{corel}, which is analogous to the same notion for  $C^*$-algebras (\cite{ght}). It is shown in \cite{ght}*{Theorem 13.1} that the $E$-theory Baum-Connes assembly map for the full $C^*$-crossed product with coefficients in proper $G$-$C^*$-algebras is an isomorphism. The analogous result for algebraic $K$-theory of algebraic crossed products of groups and $K$-excisive $\Q$-algebras and the Farrell-Jones assembly map was proved in \cite{corel}*{Theorem 13.2.1}. Higson and Kasparov show in \cite{hk2} that if the affine isometric action of $G$ on $V$ is metrically proper, then $\cA_0(V)$ is a proper $G$-$C^*$-algebra. We prove in Theorem \ref{thm:proper} that if 
\[
\tau:\F\to \G
\]
is a natural transformation between functors $G$-$\Bor\to\spt$, then the map 
$\tau(\cAc(V))$ is a weak equivalence whenever all the following conditions are satisfied: 
\begin{itemize}
\item The action of $G$ on $V$ is metrically proper.
\item The functors $\F$ and $\G$ satisfy excision and commute up to weak equivalence with filtering colimits along injective maps.
\item If $H\subset G$ is a finite subgroup and $P$ is proper over $G/H$, then $\tau(P)$ is an equivalence.
\end{itemize}
All these results are used in Section \ref{sec:main} to prove Theorem \ref{thm:intro} (for general bornolocal $C^*$-algebras) and Corollaries \ref{cor:intro1} and \ref{cor:intro2}; they are Theorem \ref{thm:main} and Corollaries \ref{cor:main1} and \ref{cor:main2}, respectively.

\goodbreak

\smallskip

\noindent{\it Acknowledgements.} We wish to thank our colleague Gabriel Acosta for useful discussions, and Arthur Bartels and the anonymous referee for pointing out mistakes in previous versions of this paper. Part of the research for this article was carried out
while the first named author was visiting Sasha Gorokhovsky at the University of Colorado Boulder. He is thankful to UCB and his host for their hospitality and to the latter for useful discussions.

\goodbreak

\section{Bornolocal \topdf{$C^*$}{C*}-algebras}\label{sec:bornoloco}
\numberwithin{equation}{subsection}
\subsection{Definitions and examples}\label{subsec:defex}

Let $(A,\left\|\right\|)$ be a normed $*$-algebra over $\C$ such that $\left\|a^*a\right\|=\left\|a\right\|^2$ for all $a\in A$. A \emph{$C^*$-bornology} for $A$ is a filtered family $\cF$ of complete $*$-subalgebras that verifies $\bigcup_{\fA\in\cF}\fA =A$. If $\cF$ and $\cF^{\prime}$ are two $C^*$-bornologies on $A$, we say that $\cF$ is \emph{finer} than $\cF^{\prime}$ (and write $\cF \prec\cF^{\prime}$) if for each $\fA\in\cF$ there exists $\fA^{\prime}\in\cF^{\prime}$ such that $\fA\subset\fA^{\prime}$. If $\cF\prec\cF^{\prime}$ and $\cF^{\prime}\prec\cF$ we call the bornologies \emph{equivalent}.
 A \emph{bornolocal $C^*$-algebra} is a normed $*$-algebra $A$ as above equipped with an equivalence class of $C^*$-bornologies. Thus a bornolocal $C^*$-algebra is a local $C^*$-algebra in the bornological sense (cf. \cite{cmr}*{Definition 2.11}).
We write $(A,\cF)$ or simply $A$ for the algebra $A$ equipped with the equivalence class of the $C^*$-bornology $\cF$, depending on whether or not the latter needs to be emphasized. 
A \emph{morphism} between two bornolocal $C^*$-algebras $(A,\cF)$ and $(B,\cG)$ is a $*$-homomorphism $f:A\to B$ such that $\cF\prec f^{-1}(\cG):=\{f^{-1}(\fB):\fB
\in\cG\}$. Note that this definition depends only on the equivalence classes of the bornologies $\cF$ and $\cG$. For example if $(A,\cF)$ is a bornolocal $C^*$-algebra and $C\subset A$ is a closed $*$-subalgebra then $C$ is again a bornolocal $C^*$-algebra with the \emph{induced bornology} 
\begin{equation}\label{induborno}
\{\fA\cap C:\fA\in\cF\}
\end{equation}
and the inclusion is a homomorphism. We write $\Bor$ for the category of bornolocal $C^*$-algebras and morphisms. 

Any $C^*$-algebra $\fA$ may be viewed as a bornolocal $C^*$-algebra with the trivial bornology $\cF=\left\{\fA\right\}$. This gives a fully faithful embedding of the category of $C^*$-algebras into $\Bor$. If $\{A_i\}$ is a filtering system of bornolocal $C^*$-algebras with injective transfer maps then the algebraic colimit $A=\colim_iA_i$, equipped with the obvious colimit bornology, is the colimit of the system in $\Bor$. Thus any functor $F:C^*$-$\alg\to C^*$-$\alg$ which preserves monomorphisms extends to bornolocal $C^*$-algebras by 
\begin{equation}\label{eq:extendfun}
F(A,\cF)=\colim_\cF F(\fA).
\end{equation}
Hence, for example, if $X$ is a locally compact (Hausdorff) space and $A\in\Bor$ then the algebra $C_0(X,A)$ of continuous functions vanishing at infinity is again in $\Bor$. Moreover the algebra of compactly supported continuous functions is also in $\Bor$, since we may write it as the colimit
\[
C_c(X,A)=\colim\ker (C(K,\fA)\to C(\partial K,\fA)).
\]
Here the colimit runs over all pairs $(\fA,K)$ with $\fA\in\cF$ and $K\subset X$ a compact subspace which is the closure of an open subset. Recall from \cite{wegge}*{T.5.19} that the spatial tensor product $\sotimes$ of injective morphisms of $C^*$-algebras is again injective. The \emph{spatial tensor product} $A\sotimes B$ of bornolocal $C^*$-algebras is defined by using \eqref{eq:extendfun} twice. For example, $C_c(X,A)=C_c(X)\sotimes A$. The \emph{graded spatial tensor product} $A\gotimes B$ of $\Z/2\Z$-graded bornolocal $C^*$-algebras $A$ and $B$ is defined similarly.

If $B\in\Bor$ we write $B[0,1]=C([0,1],B)$ for the algebra of continuous functions. Two homomorphisms $f_0,f_1:A\to B\in\Bor$ are \emph{homotopic} if there exists $H:A\to B[0,1]\in\Bor$ such that $\ev_i H=f_i$ ($i=0,1$).
 
\subsection{Exact sequences}\label{subsec:exaseqs}

If $(A,\cF)\in\Bor$ then a \emph{bornolocal ideal} in $A$ is a ring theoretic, closed two-sided ideal $I$, equipped with the equivalence class of the induced bornology \eqref{induborno}. Note that any such ideal is automatically self-adjoint.
The kernel of a homomorphism $f:A\to B$ in $\Bor$ in the categorical sense is just the ring theoretic kernel $\ker f$ with the induced bornology. If $A=(A,\cF)\in\Bor$ and $I\triqui A$ is a bornolocal ideal, then the cokernel of the inclusion map $I\subset A$ is $A/I$ equipped with the equivalence class of the bornology $\{\fA/\fA\cap I:\fA\in\cF\}$. A sequence 
\begin{equation}\label{seq:abc}
\xymatrix{0\to A\ar[r]^(.6)i& B\ar[r]^(.4)p &C\to 0}
\end{equation}
of bornolocal $C^*$-algebras is \emph{exact} if $i$ is a kernel of $p$ and $p$ is a cokernel of $i$. By our previous remarks, if $B=(B,\cF)$ then \eqref{seq:abc} is
isomorphic to the algebraic colimit of the exact sequences of $C^*$-algebras
\[
\xymatrix{0\to A\cap\fB\ar[r]^(.6)i& \fB\ar[r]^(.3)p &\fB/A\cap\fB\to 0}
\]
with $\fB\in\cF$, and this colimit coincides with the colimit in $\Bor$. Conversely, the colimit in $\Bor$ of any filtering system of short exact sequences of $C^*$-algebras along monomorphisms is exact. 

\section{Equivariant homology}\label{sec:equihomo}

\subsection{Homotopy invariance, excision, stability, and equivariant \topdf{$E$}{E}-theory}\label{subsec:eg}

Let $G$ be a countable discrete group. Consider the category\linebreak
$G$-$\Bor$ of $G$-bornolocal $C^*$-algebras and equivariant homomorphisms. If $A,B\in G$-$\Bor$,
we equip $A\sotimes B$ with the diagonal action. Let $A[0,1]=C([0,1],A)=A\sotimes \C[0,1]$ with the trivial action on $\C[0,1]$. The natural map
\[
c:A\to A[0,1],\ \ c(a)(t)=a, \ \ t\in[0,1],
\]
is $G$-equivariant. 
Let $\E:G$-$\Bor\to\spt$. We say that $\E$ is \emph{homotopy invariant} if $\E(c)$ is a weak equivalence for every $A\in G$-$\Bor$. We say that $\E$ is \emph{excisive} if for every exact sequence \eqref{seq:abc} of equivariant maps, 
\begin{equation}\label{seq:exciE}
\E(A)\to \E(B)\to \E(C)
\end{equation}
is a homotopy fibration sequence.  

Any equivariant orthogonal decomposition $H=H_1\perp H_2$ of a separable $G$-Hilbert space gives rise to a $C^*$-algebra homomorphism $\cK(H_i)\to \cK(H)$ ($i=1,2$) between the algebras of compact operators. We say that $\E$ is \emph{$G$-stable} (resp. \emph{stable}) if for every equivariant orthogonal decomposition as above (resp. for every decomposition as above where $\dim H_1=1$ and $G$ acts trivially on $H$) and every $A\in G$-$\Bor$,
$\E$ sends the maps 
\begin{equation}\label{maps:gstab}
A\sotimes \cK(H_1)\to A\sotimes \cK(H)\leftarrow A\sotimes\cK(H_2)
\end{equation} 
to weak equivalences. Thus if $H_1$ and $H_2$
are $G$-Hilbert spaces and $\E$ is $G$-stable then the maps \eqref{maps:gstab} induce a weak equivalence 
\[
\E(A\sotimes \cK(H_1))\weq \E(A\sotimes \cK(H_2)).
\]

\subsection{Equivariant algebraic \topdf{$K$}{K}-homology}
Write $K$ for the nonconnective algebraic $K$-theory spectrum. If $R$ is a ring and $I\triqui R$ is an ideal, we write $K(R:I)=\hofi(K(R)\to K(R/I))$. Recall that a ring $I$ is \emph{$K$-excisive} if whenever $I\triqui R$
and $I\triqui S$ are two ideal embeddings and $f:R\to S$ is a compatible ring homomorphism, the map $K(R:I)\to K(S:I)$ is a weak equivalence.  

If $I$ is a $G$-ring, the (algebraic) \emph{crossed product} $I\rtimes G$ is the tensor product $I\otimes \Z[G]$ equipped with the twisted product
\[
(a\rtimes g)(b\rtimes h)=ag(b)\rtimes gh. 
\]
In the next lemma and elsewhere, we shall use the \emph{maximal tensor product} $C\motimes D$ of two bornolocal $C^*$-algebras $C,D$. If $(C,\cF)$, $(D,\cG)\in\Bor$, this is the 
algebraic colimit
\begin{equation}\label{motimes}
C\motimes D=\colim_{\cF\times\cG}\fC\motimes\fD.
\end{equation}
One checks that the colimit depends only on the equivalence classes of $\cF$ and $\cG$, so that $C\motimes D$ is a well-defined $\C$-algebra. If either $C$ or $D$ is \emph{nuclear}, i.e. it has a bornology in which every element is a nuclear $C^*$-algebra, then $C\motimes D=C\sotimes D$ is the bornolocal $C^*$-algebra of Section \ref{subsec:defex}. If $G$ is a discrete group acting on both $C$ and $D$, then we consider $C\motimes D$ as a $G$-algebra equipped with the diagonal action. If $B,C$ and $D$ are in $G$-$\Bor$ and $C$ is nuclear, then there is an associativity isomorphism
\begin{equation}\label{asoc}
(B\motimes C)\motimes D\cong B\motimes (C\motimes D).
\end{equation}
In this case we shall abuse notation and write $B\motimes C\motimes D$ for $(B\motimes C)\motimes D$.

Let $G$ be a group and let $\org$ be its \emph{orbit category}. If $G/H\in \org$ write $\cG(G/H)$ for the \emph{transport groupoid}.  If $R$ is a unital $G$-ring, we can form the
\emph{crossed product} $\Z$-linear category $R\rtimes\cG(G/H)$ \cite{corel}*{Section 3.1}. Let $I\triqui R$ be a two-sided ideal, closed under the action of $G$; consider the homotopy fiber 
\begin{multline*}
K(R\rtimes\cG(G/H):I\rtimes\cG(G/H))=\\
\hofi(K(R\rtimes\cG(G/H))\to K(R/I\rtimes\cG(G/H))).
\end{multline*}
The \emph{$G$-equivariant $K$-homology} of a $G$-space $X$ with coefficients in $(R:I)$ is the coend
\[
H^G(X,K(R:I))=\int^{\org}X^H_+\land K(R\rtimes \cG(G/H):I\rtimes \cG(G/H)).
\]
Let $\tilde{I}$ be the \emph{unitalization}; this is the abelian group $I\oplus\Z$ equipped with the following multiplication:
\[
(x,m)(y,n)=(xy+my+nx,mn).
\]
If $I$ is a $G$-ring, we write
\begin{gather*}
K(I\rtimes\cG(G/H))=K(\tilde{I}\rtimes\cG(G/H):I\rtimes\cG(G/H))\\
\mbox{and }H^G(X,K(I))=H^G(X,K(\tilde{I}:I)).
\end{gather*}
If $I$ is unital, the two definitions of $K(I\rtimes\cG(G/H))$ and $H^G(X,K(I))$ are weakly equivalent, by \cite{corel}*{Propositions 3.3.9(a) and 4.3.1}. If $I$ is $K$-excisive and $I\triqui R$ is an ideal embedding into a unital ring $R$, then by \cite{corel}*{Propositions 3.3.12 and 4.3.1}, the canonical map of $\org$-spectra is a weak equivalence
\begin{equation}\label{map:exciorg}
K(I\rtimes\cG(G/H))\weq K(R\rtimes\cG(G/H):I\rtimes\cG(G/H)).
\end{equation}
For any $G$-ring $I$ we have a weak equivalence, natural in $I$
\[
K(\tilde{I}\rtimes H:I\rtimes H)\weq K(I\rtimes\cG(G/H)).
\]
If $I$ is $K$-excisive, we furthermore have
\[
K(I\rtimes H)\weq K(\tilde{I}\rtimes H:I\rtimes H).
\]
It was proved in \cite{corel} that $K(-\rtimes\cG(G/H))$ and $H^G(X,K(-))$ send short exact sequences of $K$-excisive $G$-rings to homotopy fibration sequences (\cite{corel}*{Propositions 3.3.9(b) and 4.3.1}). More generally, we have the following proposition, which we shall use in Section \ref{sec:asy}.

\begin{prop}\label{prop:exci}
Let
\[
0\to I\to A\to B\to 0
\]
be an exact sequence of $G$-rings. Assume that $I$ is $K$-excisive. Let $X$ be a $G$-space. Then 
\[
H^G(X,K(I))\to H^G(X,K(A))\to H^G(X,K(B))
\]
is a homotopy fibration sequence.
\end{prop}
\begin{proof} It suffices to prove the proposition for $X$ of the form $G/H$ where $H\subset G$ is a subgroup. 
We have a homotopy commutative diagram with homotopy fibration rows
\[
\xymatrix{
K(A\rtimes\cG(G/H))\ar[r]\ar[d]&K(\tilde{A}\rtimes\cG(G/H))\ar[d]\ar[r]&K(\Z[\cG(G/H)])\ar@{=}[d]\\
K(B\rtimes\cG(G/H))\ar[r]&K(\tilde{B}\rtimes\cG(G/H))\ar[r]&K(\Z[\cG(G/H)])\\
}
\]
It follows that the homotopy fiber of the first vertical map is weakly equivalent to that of the middle map, which in turn is weakly equivalent to $K(I\rtimes\cG(G/H))$, by \eqref{map:exciorg}.
\end{proof}

\bigskip

\subsection{Equivariant \topdf{$K$}{K}-homology of stable algebras}

The following lemma is well-known and straightforward; it will be used in the proof of Theorem \ref{thm:kstab} below.

\begin{lem}\label{lem:alma}
Let $H$ be a $G$-Hilbert space; if $g\in G$, write $u_g\in\cB(H)$ for the unitary implementing the action of $g$ on $H$. Let $I$ be a $G$-ring and $A,B\in G$-$\Bor$. Let $\underline{H}$ be $H$ with the trivial $G$-action. Then the map
\begin{align*}
(I\otimes (A\motimes \cK(H)\motimes B))\rtimes G\to (I\otimes (A\motimes\cK(\ul{H})\motimes B))\rtimes G\\ 
(x\otimes (a\motimes T\motimes b)) \rtimes g\mapsto (x\otimes (a\motimes Tu_g\motimes b))\rtimes g
\end{align*}
is an isomorphism.
\end{lem}

\bigskip

We now come to the main theorem of this section.

\begin{thm}\label{thm:kstab}
Let $G$ be a countable discrete group, $I$ a $G$-ring, $B\in G$-$\Bor$ and $\cK=\cK(\ell^2(\N))$ the algebra of compact operators with trivial $G$-action. Assume that $I$ is $K$-excisive. Let $X$ be a $G$-simplicial set. Then the functor
\[
G\emph{-}\Bor\to\spt,\ \ A\mapsto H^G(X, K(I\otimes (A\motimes \cK\motimes B)))
\]
is excisive, homotopy invariant, and $G$-stable.
\end{thm}
\begin{proof}
By \cite{sw}*{Corollary 10.4}, $C^*$-algebras are $K$-excisive, and by \cite{corel}*{Proposition A.4.4} $K$-excisive rings are closed under filtering colimits. It follows
that $A\motimes B$ is $K$-excisive for every pair of bornolocal $C^*$-algebras $A$ and $B$. Hence $I\otimes (A\motimes B)$ is $K$-excisive for every $A,B\in\Bor$, by \cite{corel}*{Proposition A.5.3}. Besides, by \cite{ght}*{Lemma 4.1} and Section \ref{subsec:exaseqs}, 
$-\motimes B:\Bor\to\C$-$\alg$ is exact. Hence the functor of the proposition is excisive, by \cite{corel}*{Propositions 3.3.9 and 4.3.1}. Fix $n\in\Z$ and consider the functor
\[
F:C^*\emph{-}\alg\to\ab,\ \ F(\fC)=H^G_n(X,K(I\otimes((A\sotimes\fC)\motimes B))).
\]
Here $\fC$ is regarded as a $G$-$C^*$-algebra with trivial action. Again by \cite{corel}*{Propositions 3.3.9 and 4.3.1}, $F$ is split-exact. Hence 
$\fC\mapsto F(\fC\sotimes \cK)$ is homotopy invariant, by Higson's homotopy invariance theorem \cite{hig}*{Theorem 3.2.2}. Specializing to $\fC=\C$,
we obtain that the functor of the proposition is homotopy invariant, excisive and stable. To prove that it is also $G$-stable, it suffices to show that if $S\subset G$ is a subgroup, then 
\[
A\mapsto K((I\otimes (A\motimes \cK\motimes B))\rtimes\cG(G/S))
\]
is $G$-stable. By \cite{corel}*{Lemma 3.2.6 and Proposition 4.2.8}  there is a weak equivalence 
\[
K((I\otimes (A\motimes \cK\motimes B))\rtimes S)\weq K((I\otimes (A\motimes \cK\motimes B))\rtimes\cG(G/S)).
\]
It is clear that $A\mapsto K((I\otimes (A\motimes \cK\motimes B))\rtimes S)$ is stable; by Lemma \ref{lem:alma} it is also $S$-stable, and therefore $G$-stable.
\end{proof}

\begin{rem}\label{rem:kstab}
Theorem \ref{thm:kstab} will be used in full generality in the proof of Theorem \ref{thm:main}. The application given in Corollary 
\ref{cor:equiK} below uses only the case $B=\C$.
\end{rem}

Consider the comparison map
\begin{equation}\label{map:kh}
K\to KH
\end{equation}
from algebraic $K$-theory to Weibel's homotopy algebraic $K$-theory \cite{kh}.

\begin{cor}\label{cor:equiK}
Let $X$ be a $G$-space. The map \eqref{map:kh} induces a weak equivalence
\[
H^G(X, K(I\otimes (A\sotimes \cK)))\to H^G(X, KH(I\otimes (A\sotimes \cK))).
\]
\end{cor}
\begin{proof} It suffices to show that the map of $\org$-spectra
\[
K((I\otimes (A\sotimes \cK))\rtimes \cG(G/H))\to KH((I\otimes (A\sotimes \cK))\rtimes \cG(G/H))
\]
is a weak equivalence. By \cite{corel}*{Propostions 4.2.8 and 5.3} this is equivalent to proving that
\[
K((I\otimes (A\sotimes \cK))\rtimes H)\to KH((I\otimes (A\sotimes \cK))\rtimes H)
\]
is an equivalence for each subgroup $H\subset G$. By \cite{kh}*{Proposition 1.5}, the map $K(R)\to KH(R)$ is an equivalence for $K$-regular $R$. Thus it suffices to show that
$(I\otimes (A\sotimes \cK))\rtimes H$ is $K$-regular. By Theorem \ref{thm:kstab}, the functor $K_*((I\otimes(-\sotimes\cK))\rtimes H)$ is homotopy invariant. It follows from this, using the argument of the proof of \cite{comparos}*{Theorem 3.4}, that $(I\otimes(A\sotimes\cK))\rtimes H$ is $K$-regular for every $A\in G$-$\Bor$.
\end{proof}

\begin{rem}\label{rem:equiK}
By \cite{balukh}*{Remark
7.4}, if $J$ is any $G$-ring, there is a canonical weak equivalence
\[
H^G(\cE(G,\fin),KH(J))\weq H^G(\cE(G,\vcyc),KH(J)).
\]
Hence in view of Corollary \ref{cor:equiK}, if $I$, $A$ and $\cK$ are as above, then the Farrell-Jones conjecture with coefficients in $J=I\otimes (A\sotimes \cK)$ is equivalent to the isomorphism conjecture for the quadruple $(G, \fin, K, J)$. 
\end{rem}

\section{Asymptotic morphisms}\label{sec:asy}
\numberwithin{equation}{subsection}
\subsection{Basic definitions}

We begin by recalling from \cite{ght} some basic definitions and facts concerning equivariant asymptotic morphisms of $C^*$-algebras. 
Let $G$ be a discrete group and let $\fB$ be a $G$-$C^*$-algebra; write $C_b([1,\infty),\fB)$ and $C_0([1,\infty),\fB)$ for the
$G$-$C^*$-algebras of bounded continuous functions and of continous functions vanishing at infinity, equipped with the induced actions. Consider the quotient
\begin{equation}\label{elQ}
Q(\fB)=C_b([1,\infty),\fB)/C_0([1,\infty),\fB).
\end{equation}
If $n\ge 0$, we write $Q^n$ for the $n$-fold composition of the functor $Q$. Let $\fA$ be another $G$-$C^*$-algebra. A $G$-equivariant \emph{$n$-asymptotic homomorphism} from $\fA$ to $\fB$ is a $G$-equivariant $*$-homomorphism $\fA\to Q^n(\fB)$. Thus a $0$-asymptotic morphism is the same as a homomorphism of $G$-$C^*$-algebras; $1$-asymptotic morphisms are simply called \emph{asymptotic morphisms}. We shall often write 
$$f:\fA-->\fB$$
for the equivariant morphism $f:\fA\to Q(\fB)$. If $f:\fA-->\fB$ is an equivariant asymptotic morphism then any set-theoretic lift $\phi:\fA\to C_b([1,\infty),\fB)$ of $f$ can be viewed as a bounded family of maps $\phi_t:\fA\to \fB$ varying continuously on $t\in [1,\infty)$ which, roughly speaking, tends to satisfy the conditions for an equivariant homomorphism as $t\to\infty$; see \cite{ght}*{Definitions 1.3 and  1.10} for details. Such a family is called an \emph{equivariant asymptotic family} representing $f$; there is a one-to-one correspondence between equivariant asymptotic homomorphisms and classes of equivariant asymptotic families up to asymptotic equivalence \cite{ght}*{Proposition 1.11}.  
An \emph{$n$-homotopy} between $G$-equivariant morphisms $f_0,f_1:\fA\to Q^n(\fB)$ is a $G$-equivariant morphism $H:\fA\to Q^n(\fB[0,1])$ such that $Q^n(\ev_i)H=f_i$ ($i=0,1$). By \cite{ght}*{Proposition 2.3}, $n$-homotopy is an equivalence relation; we write $[[\fA,\fB]]_n$ for the set of $n$-homotopy classes of $n$-asymptotic morphisms. Let $\pi:C_b([1,\infty),\fA)\to Q(\fA)$ be the projection and let 
\begin{equation}\label{map:eliota}
\iota:\fA\to C_b([1,\infty),\fA),\ \ \iota(a)(t)=a.
\end{equation}

There is a map $[[\fA,\fB]]_n\to [[\fA,\fB]]_{n+1}$ sending the class of $f:\fA\to Q^n(\fB)$ to that of $Q(f)\pi\iota:\fA\to Q^{n+1}(\fB)$; we put
\[
[[\fA,\fB]]=\colim_n[[\fA,\fB]]_n.
\] 
If $\fA$ happens to be separable, then the map $[[\fA,\fB]]_1\to [[\fA,\fB]]$ is bijective (\cite{ght}*{Theorem 2.16}). There is a category $\fQ$ whose objects are the $G$-$C^*$-algebras and where $\hom_{\fQ}(\fA,\fB)=[[\fA,\fB]]$ (\cite{ght}*{Proposition 2.12}). The composite in $\fQ$ of the classes of $f:\fA\to Q^n(\fB)$ and $g:\fB\to Q^m(\fC)$ 
is the class of $Q^n(g)f$.

\smallskip

In the next section we shall need to consider equivariant asymptotic morphisms of bornolocal $C^*$-algebras. The definition is the same as in the $C^*$-algebra case; if $A$ and $B\in G$-$\Bor$, a $G$-equivariant \emph{asymptotic morphism} from $A$ to $B$ is a $G$-equivariant morphism 
$$A\to Q(B)=C_b([1,\infty),B)/C_0([1,\infty),B).$$ 
Here $C_b([1,\infty),B)$ is the algebra of bounded continous functions with values in the normed algebra $B$. It is normed by the supremum norm, but has no obvious $C^*$-bornology; thus $Q(B)$ is just a normed $G$-$*$-algebra. As in the $C^*$-algebra case, equivariant asymptotic morphisms are in one-to-one correspondence with classes of equivariant asymptotic families up to asymptotic equivalence. The definition
of $1$-homotopy is also the same as in the $C^*$-algebra case. We do not consider $n$-asymptotic morphisms $A\to B$ for general $B\in G$-$\Bor$and $n\geq 2$.

\subsection{Applying functors to asymptotic homomorphisms}\label{subsec:apply}

We shall\\ presently show that any excisive and homotopy invariant functor from $G$-$C^*$-algebras to spectra induces a functor $\fQ\to\ho\spt$ to the homotopy category.  We begin by noting that the kernel $C_0([1,\infty),\fB)$ of the projection $\pi$ is equivariantly contractible. Hence if $\E$ is an excisive and homotopy invariant functor to spectra, then we have a natural map $\gamma_n:\E(Q^n(\fB))\to\E(\fB)$ given by
\begin{equation*}
\E(Q^n(\fB))\overset{\pi}\lweq \E(C_b([1,\infty),Q^{n-1}(\fB)))\overset{\ev_1}\to \E(Q^{n-1}(\fB))\overset{\pi}\lweq\dots\overset{\ev_1}\to \E(\fB).
\end{equation*}

Next observe that for $t\in [0,1]$ we have the commutative diagram below, where the vertical map in the middle is induced by $Q^n(\ev_t)$

\[
\xymatrix{
Q^{n+1}(\fB[0,1])\ar[d]_{Q^{n+1}(\ev_t)}& C_b([1,\infty),Q^n(\fB[0,1]))\ar[l]_(.56)\pi\ar[r]^(.6){\ev_1}\ar[d]& Q^n(\fB[0,1])\ar[d]^{Q^n(\ev_t)}\\
Q^{n+1}(\fB)& C_b([1,\infty),Q^n(\fB))\ar[l]_\pi\ar[r]^(.6){\ev_1}& Q^n(\fB)\\
}
\]
It follows that the maps $\gamma_n\E(Q^n(\ev_0))$ and $\gamma_n\E(Q^n(\ev_1)):\E(Q^n(\fB[0,1]))\to \E(\fB)$ represent the same map in $\ho\spt$. Moreover, if $f:\fA\to Q^n(\fB)$ is a homomorphism and $\iota$ is as in \eqref{map:eliota}, then we have the following commutative diagram, where the middle horizontal map is induced by $f$
\[
\xymatrix{Q(\fA)\ar[r]^{Q(f)}& Q^{n+1}(\fB)\\{C_b([1,\infty),\fA)}\ar[u]^\pi\ar[r]&{C_b([1,\infty),Q^n(\fB))}\ar[u]^\pi\ar[r]^(.6){\ev_1}&Q^{n}(\fB)\\ \fA\ar[r]_f\ar[u]^\iota& Q^n(\fB)\ar[u]^\iota\ar@{=}[ur]}
\]

Hence the maps $\gamma_{n+1}\E(Q(f)(\pi\iota))$ and $\gamma_n(\E(f))$ are the same in the homotopy category. Thus we have a well-defined map \begin{gather}\label{map:funasy}
\bar{\E}:[[\fA,\fB]]\to \ho\spt(\E(\fA),\E(\fB))\\ 
(f:\fA\to Q^n(\fB))\mapsto\gamma_n\E(f).\nonumber
\end{gather} 
One checks further that the latter map is compatible with composition, so that we have a functor 
\[
\bar{\E}:\fQ\to \ho\spt.
\]

Recall from \cite{ght}*{Theorem 6.9} that there is also an additive category $E_G$ whose objects are the $G$-$C^*$-algebras and where the homomorphisms
are given by
\[
E_G(\fA,\fB)=[[\Sigma\fA\sotimes\cK\sotimes\cK(\ell^2(G)),\Sigma\fB\sotimes\cK\sotimes\cK(\ell^2(G))]].
\]
Here $\cK=\cK(\ell^2(\N))$ and $\Sigma\fA=C_0(\R)\sotimes\fA$ is the \emph{suspension}. There is a functor\hfill
\begin{equation}\label{map:qtoe}
\fQ\to E_G
\end{equation} 
which is the identity on objects and on morphisms is induced by tensor product with $C_0(\R)\sotimes\cK\sotimes\cK(\ell^2(G))$ (see \cite{ght}*{Theorem 4.6}). We remark that
if $\E:G$-$\fC^*\to \spt$ is excisive, homotopy invariant, and $G$-stable, then 
\[
\E(\Sigma\fA\sotimes\cK\sotimes\cK(\ell^2(G)))\weq \Omega\E(\fA).
\]
Hence we can further extend $\bar{\E}$ to a functor 
\[
\overset{=}{\E}:E_G\to\ho\spt.
\]

\smallskip

\subsection{The case of equivariant \topdf{$K$}{K}-homology}

Let $X$ be a $G$-space, $C\in G$-$\Bor$ and $I$ an excisive $G$-ring. By Theorem \ref{thm:kstab}, the functor
\begin{equation}\label{fun:ele}
\E:G\emph{-}\Bor\to\spt,\quad \E(A)=H^G(X,K(I\otimes(A\motimes\cK\motimes C)))
\end{equation}
is excisive, homotopy invariant, and $G$-stable. Hence its restriction to $G$-$C^*$-algebras induces functors $\bar{\E}:\fQ\to\ho\spt$ and $\overset{=}{\E}:E_G\to \ho\spt$. 

If we apply \eqref{map:funasy} to an equivariant homomorphism $f:\fA\to Q(\fB)$, then for 
\[
\F(\fA)=H^G(X,K(I\otimes\fA))
\]
we obtain the class of the composite

\begin{equation}\label{map1}
\bar{\E}(f):\xymatrix{\F(\fA\motimes\cK\motimes C)\ar[r]^(.45){\F(f\motimes 1)}& \F(Q(\fB)\motimes\cK\motimes C)\\
\F(\fB\motimes\cK\motimes C)&\F(C_b([1,\infty),\fB)\motimes\cK\motimes C).\ar[u]^{\F(\pi\motimes 1)}_{\wr}\ar[l]^(0.6){\F(\ev_1\motimes 1)}
}
\end{equation}

Next observe that for any $G$-$C^*$-algebra $\fD$, we have an equivariant map of exact sequences

\begin{equation}\label{map:seqs}
\xymatrix{
0\to C_0([1,\infty),\fB)\motimes\fD\ar[d]^{\wr}\ar[r]&C_b([1,\infty),\fB)\motimes\fD\ar[d]^q\ar[r]&Q(\fB)\motimes\fD\ar[d]^p\to 0\\
0\to C_0([1,\infty),\fB\motimes\fD)\ar[r]&C_b([1,\infty),\fB\motimes\fD)\ar[r]&Q(\fB\motimes\fD)\to 0
}
\end{equation}
If $\fB\in G$-$C^*$ is nuclear and $(D,\cF)\in G$-$\Bor$ then $\fB\motimes D\in G$-$\Bor$, so $Q(\fB\motimes D)$ is defined. Taking the colimit of the diagrams \eqref{map:seqs} for $\fD\in\cF$ we obtain a commutative diagram with exact rows
\begin{equation}\label{map:seqs2}
\xymatrix{
0\to C_0([1,\infty),\fB)\motimes D\ar[d]^{\wr}\ar[r]&C_b([1,\infty),\fB)\motimes D\ar[d]^q\ar[r]&Q(\fB)\motimes D\ar[d]^p\to 0\\
0\to C_0([1,\infty),\fB\motimes D)\ar[r]&C_b([1,\infty),\fB\motimes D)\ar[r]&Q(\fB\motimes D)\to 0
}
\end{equation}

In particular $\F(\pi):\F(C_b([1,\infty),\fB\motimes\cK\motimes C))\to \F(Q(\fB\motimes\cK\motimes C))$ is a weak equivalence since $C_0([1,\infty),\fB)$ is contractible and
$\F(C_0([1,\infty),\fB\motimes\cK\motimes C))\cong\E(C_0([1,\infty),\fB))$. Hence we may also consider the composite 
\begin{equation}\label{map2}
\xymatrix{\F(\fA\motimes\cK\motimes C)\ar[r]^(0.45){\F(p(f\motimes 1))}& \F( Q(\fB\motimes\cK\motimes C))\\
\F(\fB\motimes\cK\motimes C)&\F(C_b([1,\infty),\fB\motimes\cK\motimes C)).\ar[u]^{\F(\pi)}_{\wr}\ar[l]^(0.6){\F(\ev_1)}
}
\end{equation}

\begin{lem}\label{lem:mismap}
The maps \eqref{map1} and \eqref{map2} belong to the same class in $\ho\spt$.
\end{lem}
\begin{proof}
Let $D=\cK\motimes C$. The lemma follows from \eqref{map:seqs2} and from the following commutative diagram
\[
\xymatrix{C_b([1,\infty),\fB)\motimes D\ar[r]^(.6){ev_1\motimes 1}\ar[dr]_q&\fB\motimes D\\
&C_b([1,\infty),\fB\motimes D)\ar[u]_{\ev_1}.}
\]
\end{proof}

The following proposition summarizes our previous discussion.

\begin{prop}\label{prop:asy}
Let $G$ be a discrete group, $X$ a $G$-space, $C\in G$-$\Bor$, and $I$ a $K$-excisive $G$-ring. Consider the functors
\begin{gather*}
\F:G\emph{-}\alg\to\spt,\\
\F(A)=H^G(X,K(I\otimes A)),\\
\E:G\emph{-}\Bor\to\spt, \\
\E(A)=\F (A\motimes\cK\motimes C).
\end{gather*}
Then the restriction of $\E$ to the category of $G$-$C^*$-algebras induces functors $\bar{\E}:\fQ\to\ho\spt$ and $\overset{=}{\E}:E_G\to\ho\spt$ from the equivariant asymptotic category and equivariant $E$-theory to the homotopy category of spectra. The diagram
\[
\xymatrix{&E_G\ar[dr]^{\overset{=}{\E}}&\\
         \fQ\ar[ur]\ar[rr]_{\bar{\E}}&&\ho\spt\\
				G\emph{-}\fC^*\ar[u]\ar[rr]_\E&& \spt\ar[u]}
\]
commutes up to natural equivalence. If $f:\fA--->\fB$ is an equivariant asymptotic homomorphism, then $\bar{\E}(f)$ is the homotopy class of the composite of diagram \eqref{map1}. If moreover $\fB$ is nuclear, then the class of the latter map is the same as that of the composite of diagram \eqref{map2}.
\end{prop}
\begin{proof}
We showed in Section \ref{subsec:apply} that any excisive and homotopy invariant functor $G$-$\fC^*\to \spt$ extends to a functor $\fQ\to\ho\spt$, and moreover to $E_G\to\ho\spt$ if in addition the functor is $G$-stable. By Theorem
\ref{thm:kstab}, this applies to the functor $\E$. The equivalence between the maps \eqref{map1}
and \eqref{map2} is established by Lemma \ref{lem:mismap}.
\end{proof}

\smallskip
\begin{ex}\label{ex:tilde}
Let $\fA$ be a $C^*$-algebra. For $a\in C_0(\R,\fA)$ and $t\in [1,\infty)$, put
\begin{equation}\label{map:phi0}
\phi_0:C_0(\R,\fA)\to C_b([1,\infty),C_0(\R,\fA)),\hspace{0.30 cm} \phi_0(a)(t)(x)=a(x/t).
\end{equation}
Let $\fB$ be another $C^*$-algebra and let 
\[
f_0:C_0(\R,\fA)\to \fB
\]
be a $*$-homomorphism. Consider the map
\[
\hat{f_0}:C_0(\R,\fA)\to Q(\fB),\quad \hat{f_0}(a)=\pi(f_0 \phi_0(a)). 
\]
Assume that $\fB$ is nuclear, and let $C\in\Bor$. Then under the isomorphism
\[
C_0(\R,\fA)\motimes C\cong C_0(\R,\fA\motimes C),
\]
the map $p(\hat{f_0}\motimes 1)$ identifies with $\widehat{f_0\motimes 1}$. Thus if 
$\fA,\fB\in G$-$\fC^*$, $C\in G$-$\Bor$, and $\hat{f_0}$ is $G$-equivariant, then, writing 
$1$ for the identity map of $\cK\motimes C$, we have that $\widehat{f_0\motimes 1}$
is $G$-equivariant, and $\bar{\E}(\hat{f_0})$ is equivalent to the composite
\begin{equation}\label{map3}
\xymatrix{\F(C_0(\R,\fA\motimes\cK\motimes C))\ar[r]^(0.5){\F(\widehat{f_0\motimes 1})}& \F( Q(\fB\motimes\cK\motimes C))\\
\F(\fB\motimes\cK\motimes C)&\F(C_b([1,\infty),\fB\motimes\cK\motimes C)).\ar[u]^{\F(\pi)}_{\wr}\ar[l]^(0.57){\F(\ev_1)}
}
\end{equation}
\end{ex}

Now let $A,B,C\in G$-$\Bor$ with $B$ nuclear. Formula \eqref{map:phi0} defines a homomorphism $C_0(\R,A)\to C_b([1,\infty),C_0(\R,A))$, which restricts to
\[
\phi_c:C_c(\R,A)\to C_b([1,\infty),C_c(\R,A)).
\]
Let $\#\in\{0,c\}$ and let $f_\#:C_\#(\R,A)\to B$ be a $*$-homomorphism. Put 
\begin{equation}\label{map:ftilde}
\hat{f_\#}:C_\#(\R,A)\to Q(B),\quad \hat{f_\#}(a)=\pi(f_\#\phi_\#(a)).
\end{equation}
Assume that $\hat{f_\#}$ is $G$-equivariant; write $1$ for the identity map of $\cK\motimes C$. Then $\widehat{f_\#\motimes 1}$ is again $G$-equivariant. Moreover, by Proposition \ref{prop:exci} and Theorem \ref{thm:kstab}, the map
\[
\F(\pi):\F(C_b([1,\infty),B\motimes\cK\motimes C))\to \F( Q(B\motimes\cK\motimes C))
\]
is a weak equivalence. 
Let
$\bar{\E}(\hat{f_\#})$ be the composite
\begin{equation}\label{map4}
\bar{\E}(\hat{f_\#}):\xymatrix{\F(C_\#(\R,A\motimes\cK\motimes C))\ar[r]^(0.55){\F(\widehat{f_\#\motimes 1})}& \F( Q(B\motimes\cK\motimes C))\\
\F(B\motimes\cK\motimes C)&\F(C_b([1,\infty),B\motimes\cK\motimes C)).\ar[u]^{\F(\pi)}_{\wr}\ar[l]^(0.6){\F(\ev_1)}
}
\end{equation}

In the next section we shall need the following trivial observation.

\begin{lem}\label{lem:conmuta}
Let $i:A\to A'$, $j:B\to B'\in G$-$\Bor$ and let $f_c:C_c(\R,A)\to B$ and $f_0:C_0(\R,A')\to B'$ be $*$-homomorphisms. Assume that $B$ and $B'$ are nuclear and that
the diagram 
\[
\xymatrix{
C_0(\R,A')\ar[r]^(.6){f_0}& B'\\
C_c(\R,A)\ar[u]^i\ar[r]_(.6){f_c}& B\ar[u]_j
}
\]
commutes. Further assume that $\hat{f_0}$ and $\hat{f_c}$ are $G$-equivariant. Let $\E$ be as in Proposition \ref{prop:asy} and let 
$\bar{\E}(\hat{f_\#})$ be as in \eqref{map4} ($\#\in\{0,c\}$). Then the diagram
\[
\xymatrix{
\E(C_0(\R,A'))\ar[r]^(.6){\bar{\E}(\hat{f_0})}& \E(B')\\
\E(C_c(\R,A))\ar[u]^{\E(i)}\ar[r]_(.6){\bar{\E}(\hat{f_c})}& \E(B)\ar[u]_{\E(j)}
} 
\]
is homotopy commutative.
\end{lem}

\section{A Dual Dirac element}\label{sec:dirac}

\numberwithin{equation}{section}

The purpose of this section is to prove a compactly supported variant of a theorem of Higson and Kasparov (\cite{hk2}*{Theorem 6.10}). We start by recalling some material from \cite{hk1}, \cite{hk2}, and \cite{hkt}. A \emph{Euclidean space} is a real pre-Hilbert space. Let $V$ be a countably infinite dimensional Euclidean space. Write $\cF(V)$ for the set of finite dimensional affine subspaces of $V$. For $S \in \cF(V)$ put $S^0=\left\{s_1 - s_2 : s_i \in S\right\}$. Write $\Cliff(S)$ for the complexified Clifford algebra of $S^0$. As usual we use the subscripts $c$ and $0$ to indicate compactly supported functions and functions vanishing at infinity. For $\#\in\{c,0\}$, put
\[
\cC_\#(S)=C_\#(S,\Cliff(S)).
\]
Observe that the $\Z/2\Z$-grading on $\Cliff(S)$ induces one on $\cC_\#(S)$. For example $\Cliff(\R)=\C\oplus u\C$ where $u$ is a degree one element satisfying $u^2=1$. Thus 
\[
\cC_\#(\R)=C_\#(\R)\oplus u C_\#(\R).
\]
We regard $\cC_\#(\R)$ as a $\Z/2\Z$-graded algebra with homogeneous components $\cC_\#(\R)_j=u^jC_\#(\R)$  ($j=0,1$). In addition the algebra $C_\#(\R)$ is also $\Z/2\Z$-graded according to even and odd functions. For $f\in C_\#(\R)$ write $f=f^\even+f^\odd$ for its even-odd decomposition. One checks that the map
\begin{equation}\label{map:elteta}
\theta:C_\#(\R)\to \cC_\#(\R), \hspace{0.5 cm}\theta(f)=f^{\even}+uf^{\odd}
\end{equation}
is a homogeneous isometric embedding. Let $X\in C(\R)$ be the identity function. We may interpret $\theta$ as the functional calculus of the degree one, essentially self-adjoint, unbounded operator of multiplication by $Xu\in C(\R,\Cliff(\R))$; we have
\begin{equation}\label{funcalc}
\theta(f)=f(Xu).
\end{equation} 
We will identify $C_\#(\R)=\theta(C_\#(\R))$.
Consider the graded spatial tensor product
\begin{equation}\label{A(S)}
\cA_\#(S)=C_\#(\R)\gotimes\cC_\#(S).
\end{equation}
Using the identification above, we may regard $\cA_\#(S)$ as a subalgebra of $C_\#(\R\times S,\Cliff(\R\oplus S^0))$. We have
\begin{multline}\label{ida(s)}
\cA_\#(S)=\\
\{f=f^0+uf^1\in C_\#(\R\times S,\Cliff(\R\oplus S^0)): f^i(-t,s)=(-1)^if^i(t,s)\}.
\end{multline}
If $S_1 \subset S_2 \in \cF(V)$, define $S_{21}=S_2^0\ominus S_1^0$ and write $S_2=S_1+ S_{21}$. Then 
$\cA_\#(S_2)= \cA_\#(S_{21})\gotimes\cC_\#(S_1)$. Following \cite{hkt}, we write $C_{21}:S_{21}\to \Cliff(S_{21})$ for the inclusion and 
$X\in C(\R)$ for the identity function, considered as degree one, essentially self-adjoint, unbounded multipliers of $\cC_0(S_{21})$ and $C_0(\R)$, with domains $\cC_c(S_{21})$ and $C_c(\R)$. Using functional calculus, one obtains a map
\begin{equation}\label{map:elbeta}
\beta_{21}: \cA_0(S_1) \to \cA_0(S_2), \hspace{0.5 cm}\beta_{21}(f\gotimes g)=f(X\gotimes 1+1\gotimes C_{21})\gotimes g.
\end{equation}

\begin{lem}\label{lem:betamaps}
Let $v\in S_1\subset S_2\in\cF(V)$, $\rho>0$, $f\in \cAc(S_1)$ with 
$\supp(f)\subset D^1((0,v),\rho)$, the closed ball in $\R\times S_1$. Then 
$\supp(\beta_{21}(f))\subset D^2((0,v),\rho)$, the closed ball in $\R\times S_2$. In particular the map \eqref{map:elbeta} sends $\cAc(S_1)$ to $\cAc(S_2)$. 
\end{lem}
\begin{proof} It follows from the fact that if $s_2$ decomposes as $s_2=s_1+s_{21}\in S_1+S_{21}$ then 
\begin{equation}\label{formubeta}
\beta_{21}(f)(x,s_2)=f(xu+s_{21},s_1),
\end{equation}
and that for each $x$, the spectrum of $xu+s_{21}$ is $\{\pm \sqrt{x^2+||s_{21}||^2}\}$. \end{proof}

\begin{rem}\label{rem:betamono}
It follows from \eqref{map:elteta}, \eqref{funcalc}, and \eqref{formubeta}, that the map \eqref{map:elbeta} is injective. 
\end{rem}

By \cite{hkt}*{Proposition 3.2}, if $S_1\subset S_2\subset S_3$, then $\beta_{31}=\beta_{32}\beta_{21}$. 
Let $\cA_0(V)$ be the $C^*$-algebra colimit of the direct system $\{\beta_{TS}:\cA_0(S)\to \cA_0(T)\}$. Also let
\[
\cAc(V)=\colim_{\cF(V)}\cAc(S)
\]
be the algebraic colimit; by Remark \ref{rem:betamono} this is the colimit in $\Bor$. 
We have a commutative diagram
\begin{equation}\label{diag:betas}
\xymatrix{
\cA_0(0)\ar[r]^{\beta_0}&\cA_0(V)\\
\cAc(0)\ar[u]\ar[r]_{\beta_c}&\cAc(V)\ar[u]
}
\end{equation}

Now let $G$ be a discrete group acting on $V$ by affine isometries. Then for each $g\in G$ there are an orthogonal transformation $\ell(g)$ and a vector $\tau(g)\in V$ such that 
\begin{equation}\label{ellytau}
g\cdot v=\ell(g)(v)+\tau(g)\qquad (v\in V).
\end{equation}
The $G$-action on $V$ induces an action on $\cA_\#(V)$ defined as follows
\[
(g\cdot f)(v)=\ell(g)(f(g^{-1}\cdot v)).
\]
We regard $\cA_\#(0)$ and $\cA_\#(V)$ as $G$-algebras with the trivial and the induced action, respectively. 

In general, the map $\beta_\#:\cA_\#(0)\to \cA_\#(V)$ is not $G$-equivariant; however this can be fixed asymptotically. Indeed the asymptotic homorphism
\[
\hat{\beta}_\#:\cA_\#(0)--->\cA_\#(V)
\]
defined by \eqref{map:ftilde} is $G$-equivariant.

The following proposition is an immediate consequence of a theorem of Higson and Kasparov and of the results of the previous section.

\begin{prop}\label{prop:bott} (cf. \cite{hk2}*{Theorem 6.8}).
Let $G$ be a countable discrete group acting on $V$ by affine isometries. Let $X$ be a $G$-space, let $I$ be a $K$-excisive $G$-ring, and let $B\in G$-$\Bor$. Consider the functor
\[
\E_X:G\emph{-}\Bor\to\spt, \quad \E_X(A)=H^G(X,K(I\otimes (A\motimes\cK\motimes B))).
\] 
Then the map $\bar{\E}_X(\hat{\beta_{c}})$ defined in \eqref{map4} is a split monomorphism in $\ho\spt$.
\end{prop}
\begin{proof}
Put $\E=\E_X:G$-$\Bor\to\spt$. By Proposition \ref{prop:asy} and \cite{hk2}*{Theorems 6.8 and 6.11}, the functor \eqref{map:qtoe} sends the class of $\hat{\beta_0}$ to an isomorphism in $E_G$. Hence in view of \eqref{diag:betas} and of Lemma \ref{lem:conmuta} it suffices to show that $\E$ sends the inclusion $C_c(\R)=\cAc(0)\to \cA_0(0)=C_0(\R)$ to a weak equivalence. Because $\E$ commutes up to weak homotopy equivalence with filtering colimits along injective maps, the natural map $\colim_{\rho>0}\E(C_0(-\rho,\rho))\to \E(C_c(\R))$ is a weak equivalence. For each $\rho>0$, $C_0(-\rho,\rho)\triqui C_0(\R)$ is an ideal 
and the quotient $C_0(\R)/C_0(-\rho,\rho)\cong C_0((-\infty,-\rho]\cup[\rho,\infty))$ is contractible; indeed $H(f)(s,t)=f(t/s)$ is a contraction. 
Thus because the functor $\E$ is homotopy invariant and excisive, $\E(C_0(-\rho,\rho))\to \E(C_0(\R))$ is a weak equivalence. Hence we have a weak equivalence
\begin{equation}\label{map:ccc0}
\E(C_c(\R))\weq \E(C_0(\R)).
\end{equation}
This finishes the proof. 
\end{proof}

\begin{cor}\label{cor:bott}
Let $Y$ be another $G$-space, and $f:X\to Y$ an equivariant map. Let $\tau:\E_X\to \E_Y$ be the natural map induced by $f$. Assume that $\tau(\cAc(V))$ is a weak equivalence. Then $\tau(\C)$ is a weak equivalence too.
\end{cor}
\begin{proof}
By excision and homotopy invariance, $\tau(\C)$ is equivalent to the delooping of $\tau(C_0(\R))$ in $\ho\spt$. By \eqref{map:ccc0} the latter map is equivalent to $\tau(C_c(\R))$. The corollary now follows from the proposition above and the fact that a retract of an isomorphism is an isomorphism.
\end{proof}

\section{Proper actions}\label{sec:proper}
Let $G$ be a discrete group. If $J\in G$-$\ring$ is commutative but not necessarily unital and $I\in G$-$\ring$, then by a {\em compatible $(G,J)$-algebra structure} on
$I$ we understand a $J$-bimodule structure on $I$ such
that the following identities hold for $a,b\in I$, $c\in J$, and $g\in
G$:
\begin{equation}\label{condipropertriv}
\begin{gathered}
c\cdot a=a\cdot c,\\
c\cdot(ab)=(c\cdot a)b=a(c\cdot b),\\
g(c\cdot a)=g(c)\cdot g(a).\\
\end{gathered}
\end{equation}
If $I$ and $J$ are $*$-$\C$-algebras we will additionally require the following two conditions
\begin{equation}\label{condiproperstar}
(\lambda c)\cdot a=c\cdot (\lambda a),\ \ 
(c\cdot a)^*=c^*\cdot a^* \quad (\lambda\in \C,\ \ c\in J,\ \ a\in I).
\end{equation}
If moreover $I$ and $J$ are normed, we will further ask that
\begin{equation}\label{condinorm}
||c\cdot a||\le ||c||||a||,\quad (c\in J,\ \ a\in I).
\end{equation}
We say that a compatible $(G,J)$-algebra structure is \emph{full} if it satisfies the additional
condition
\begin{equation}\label{condifull}
J\cdot I=I.
\end{equation}
If $(A,\cF),(B,\cG)\in G$-$\Bor$ with $A$ commutative, for a compatible $(G,A)$-algebra structure on $B$ to be \emph{full} we shall also require that $\cG$ be equivalent to the filtration
$\cF\cdot\cG$ consisting of the $*$-subalgebras $\fA\cdot\fB$ with $\fA\in\cF$ and $\fB\in\cG$:
\begin{equation}\label{condibornofull}
\cF\cdot\cG\sim \cG.
\end{equation}
Let $H\subset G$ be a subgroup. The ring
\[
\Z^{(G/H)}=\{f:G/H\to \Z: |\supp(f)|<\infty\}=\bigoplus_{gH\in G/H}\Z
\]
has a natural $G$-action. We say that a $G$-ring $I$ is \emph{proper} over $G/H$ if it carries
a full compatible $(G,\Z^{(G/H)})$-structure. Observe that 
\begin{equation}\label{ccgh}
\C^{(G/H)}=C_c(G/H)\in G\emph{-}\Bor
\end{equation}
is the algebra of compactly supported continuous functions. We say that $A\in G$-$\Bor$ is \emph{proper} over $G/H$ if it carries a full compatible $(G,\C^{(G/H)})$-algebra structure. Then if $x\in G/H$ and $\chi_x$ is the characteristic function, \eqref{condiproperstar} implies that multiplication by $\chi_x$ is a $*$-homomorphism $A\to A$ with image $A_x=\chi_xA$. Hence $A_x$
is a closed $*$-subalgebra, and we have a direct sum decomposition
\begin{equation}\label{propersum}
A=\bigoplus_{x\in G/H}A_x
\end{equation}
where each $A_x\in\Bor$, and $A_H\in H$-$\Bor$. If $\cF$ is a bornology in the equivalence class of $A$, then the induced $C^*$-bornology in $A_x$ consists of the $C^*$-algebras $\fA_x=\fA\cap A_x$ with $\fA\in\cF$. The algebra $A$ also carries the following $C^*$-bornology
\[
\cF^{\bu}=\{\bigoplus_{x\in F}\fA_x: F\subset G/H \text{ finite}, \fA\in\cF\}.
\]
Condition \eqref{condibornofull} implies that $\cF^{\bu}$ is equivalent to $\cF$:
\begin{equation}\label{bornofullproper}
\cF\sim\cF^{\bu}.
\end{equation}

\begin{rem}\label{rem:prop-prop}
By \eqref{propersum} and \eqref{bornofullproper}, if $G/H$ is infinite and $A$ is nonzero and proper over $G/H$, then 
$A$ cannot be complete, since it is isomorphic as a bornolocal $C^*$-algebra to an infinite algebraic direct sum of copies of $A_H$. In particular, no nonzero $G$-$C^*$-algebra can be proper in our sense over an infinite homogeneous space $G/H$.  
\end{rem}

\begin{lem}\label{lem:propertenso}
Let $A,B\in G$-$\Bor$. Assume that $A$ is proper over $G/H$. Then $A\sotimes B$ and $A\motimes B$ are proper over $G/H$, as a $G$-bornolocal $C^*$-algebra in the first case, and as a $G$-$*$-algebra in the second.
\end{lem}
\begin{proof} Straightforward.
\end{proof}

\begin{lem}\label{lem:iprop-jprop}
Let $A,B\in G$-$\Bor$ with $A$ commutative and let $H\subset G$ be a subgroup. Assume that $A$ is proper over $G/H$ and that $B$ is equipped with a full compatible $(G,A)$-algebra structure.
Then $B$ is proper over $G/H$.
\end{lem}
\begin{proof}
For $x\in G/H$ let $B_x=A_xB$. It follows from \eqref{propersum} that $B=\sum_xB_x$. Next we show that $B_x\cap B_y=0$ if $x\ne y$. Let $b\in B_x\cap B_y$. Then there exist $n$, $a_1,\dots,a_n\in A_x$ and $b_1,\dots,b_n\in B$ such that
\begin{equation}\label{bsuma}
b=\sum_{i=1}^na_ib_i.
\end{equation} 
Because $A_x$ is a bornolocal $C^*$-algebra, there is a $C^*$-subalgebra $\fA_x\subset A_x$, such
that $a_1,\dots,a_n\in \fA_x$. Let $\{e_\lambda\}$ be a bounded approximate unit in $\fA_x$.
Use \eqref{condinorm} and \eqref{bsuma} to show that $\lim_\lambda e_\lambda b=b$. On the other hand, $e_\lambda\in A_x$ and $b\in B_y$ implies $e_\lambda b=0$. Hence $B_x\cap B_y=0$, as claimed. Define an action of $\C^{(G/H)}$ on $B$ as follows. For $c=\sum_x\lambda_x\chi_x\in \C^{(G/H)}$ and $b=\sum_xb_x\in B$, put
\[
c\cdot b=\sum_x\lambda_xb_x.
\]
One checks that this action satisfies \eqref{condipropertriv}, \eqref{condiproperstar} and \eqref{condinorm}. Moreover \eqref{condibornofull} and \eqref{bornofullproper} together imply
that if $B=(B,\cG)$ then $\cG^\bu\sim\cG$. Thus $B$ is proper over $G/H$. 
\end{proof}

Let $G$ be a countable discrete group and $V$ a Euclidean space of countably infinite dimension where $G$ acts by affine isometries. 
We say that the action of $G$ on $V$ is \emph{metrically proper} if  
\begin{equation}\label{metprop}
\lim_{g \to \infty}\left\|g\cdot v\right\|=\infty \qquad(\forall v \in V).
\end{equation}
The condition that a group $G$ admits such an action is the \emph{Haagerup approximation property}. In the literature, the groups that have this property are sometimes called \emph{a-T-menable groups} and sometimes \emph{Haagerup groups}. 

The purpose of this section is to prove the following.

\begin{thm}\label{thm:proper}
Let $G$ be a countable discrete group and let $V$ be a Euclidean space of countably infinite dimension with an action of $G$ by affine isometries. 
Let $\E,\F:G$-$\Bor\to\spt$ be functors and $\tau:\E\to\F$ a natural transformation. Assume:
\item[i)] The action of $G$ on $V$ is metrically proper.
\item[ii)] If $H\subset G$ is a finite subgroup and $P\in G$-$\Bor$ is proper over $G/H$, then $\tau(P)$ is a weak equivalence.
\item[iii)] The functors $\E$ and $\F$ are excisive and commute with filtering colimits along injective maps up to weak equivalence. 

\goodbreak

Then the map $\tau(\cA_c(V))$ is a weak equivalence. 
\end{thm}

The proof of Theorem \ref{thm:proper} will be given at the end of the section. First we need to introduce some notation and prove some lemmas. 
Let $S\in \cF(V)$; consider the subalgebra
\[
\cZc(S)=\{f\in C_c(\R\times S): f(-t,s)=f(t,s)\}\subset\cAc(S).
\] 
Observe that $\cZc(S)$ lies in the center of $\cAc(S)$ and, moreover, we have
\begin{equation}\label{unicond}
\cAc(S)=\cZc(S)\cAc(S).
\end{equation}
Write
\[
\R_+=[0,\infty).
\]
Restriction along the inclusion $\R_+\times S\subset \R\times S$ induces an isomorphism
\begin{equation}\label{iso:cac(s)}
\cZc(S)\cong C_c(\R_+\times S).
\end{equation}
>From now on we shall identify both sides of \eqref{iso:cac(s)}. Let $S\subset T\in\cF(V)$; every element of $T$ writes uniquely
as 
\[
t=\pi_S(t)+\pi_S^{\perp}(t)\quad \pi_S(t)\in S,\ \ \pi_{S}^\perp(t)\in T^0\ominus S^0. 
\] 
Consider the map
\begin{gather}
p_{ST}:\R_+\times T\to \R_+\times S,\nonumber\\
         p_{ST}(x,t)=(\sqrt{x^2+||\pi_{S}^\perp(t)||^2},\pi_S(t)).\label{map:pi}
\end{gather}
In view of \eqref{formubeta}, under the isomorphism of
\eqref{iso:cac(s)}, the restriction of $\beta_{TS}$ to $\cZc(S)$ identifies with composition with $p_{ST}$:
\begin{equation}\label{map:newbeta}
\beta_{TS}(f)=fp_{ST}. 
\end{equation}
Put
\[
\cZc(V)=\colim_S\cZc(S).
\]
Consider the inverse system of locally compact topological spaces and proper maps $\{p_{ST}:\R_+\times T\to \R_+\times S\}$. Let
$\fH=\ol{V}$ be the Hilbert space completion; write $\fH_w$ for $\fH$ equipped with the locally convex topology of weak convergence. Equip 
\begin{equation}\label{elfx}
\fX:=\R_+\times \fH
\end{equation}
with the coarsest topology such that both the projection $\R_+\times \fH\to \fH_w$ and the map $\R_+\times \fH\to \R_+$, $(x,h)\mapsto \sqrt{x^2+||h||^2}$ are continuous. If $h\in \fH$ and $S\in\cF(V)$, write $h=\pi_S(h)+\pi_S^{\perp}(h)\in S+S_0^\perp$. Let
\[
p_S:\fX\to \R_+\times S,\ \ p_S(x,h)=(\sqrt{x^2+||\pi_S^{\perp}(h)||^2},\pi_S(h)).
\]
We have a homeomorphism
\begin{equation}\label{map:homeo}
\fX\to \lim_{S\in\cF(V)}\R_+\times S,\ \ (x,h)\mapsto (p_S(x,h))_S.
\end{equation}
Observe that if $S\in\cF(V)$, then the subspace topology on $\R_+\times S\subset\fX$ is the usual Euclidean topology. 
Let $v\in S\in\cF(V)$ and let
\[
\Da_S((r,v),\delta)=\{(x,s):(x-r)^2+||s-v||^2<\delta^2\}
\]
be the open ball in $\R_+\times S$. The subsets
\begin{equation}\label{elU}
U(S,r,v,\delta)=p_S^{-1}(\Da_S((r,v),\delta))\quad (S\in\cF(V), (r,v)\in \R_+\times S, \delta>0),
\end{equation}
are open and form a sub-basis for the topology of $\fX$. 
Observe that the maps 
\[
\cZc(S)\to C_c(\fX),\quad f\mapsto fp_S
\]
induce a monomorphism 
\begin{equation}\label{map:embed}
\cZc(V)\hookrightarrow C_c(\fX).
\end{equation}
Its image consists of those $f$ which factor through a projection $p_S$. 

\goodbreak

Let $S\in\cF(V)$, $X\subset \R_+\times S$ a locally closed subset. Put
\[
\cZc(X)=C_c(X).
\]
If $X$ happens to be open then $\cZc(X)$ is the subalgebra of $\cZc(S)$ consisting of those elements $f$ with $\supp(f)\subset X$.

\begin{lem}\label{lem:tietze}
Let $S\in\cF(V)$, $X\subset \R_+\times S$ a locally closed subset, and let $Z\subset X$ be closed in the subspace topology. Then the restriction map
$\cZc(X)\to \cZc(Z)$ is surjective.
\end{lem}
\begin{proof} This is a straightforward application of Tietze's extension theorem.
\end{proof}

For $X\supset Z$ as in Lemma \ref{lem:tietze}, we write
\[
I(X,Z)=\ker(\cZc(X)\to \cZc(Z)).
\]

The following trivial observation will be useful in what follows.

\begin{lem}\label{lem:pro}
Let $S\subset T\in\cF(V)$ and $X\subset \R_+\times S$. Then $p_S^{-1}(X)\cap (\R_+\times T)=p_{ST}^{-1}(X)$. 
\end{lem}

Let $S\in\cF(V)$ and let $X\subset \R_+\times S$ be a locally closed subset. Put $L=p_S^{-1}(X)$; by Lemma \ref{lem:pro}, if $T'\supset T\supset S$, then \eqref{map:newbeta} defines a map $\beta_{T'T}:\cZc((\R_+\times T)\cap L)\to \cZc((\R_+\times T')\cap L)$. Write
\begin{equation}\label{acl}
\cZc(L)=\colim_{T\supset S}\cZc((\R_+\times T)\cap L).
\end{equation}
If $Z\subset X$ is closed in the subspace topology, and $M=p_S^{-1}(Z)$, we write $I(L,M)=\ker(\cZc(L)\to\cZc(M))$. We have 
\begin{equation}\label{ilm}
I(L,M)=\colim_{T\supset S}I((\R_+\times T)\cap L,(\R_+\times T)\cap M).
\end{equation}

If now $G$ acts on $V$ by affine isometries, then the action extends by continuity to an action by affine isometries on $\fH$.
Let $G$ act on 
$\fX$ via $g(x,h)=(x,gh)$. We also have an action of $G$ on $\lim_S(\R_+\times S)$ via 
$$
(g(x_S,v_S))_{gS}=(x_S,g(v_S));
$$ 
one checks that the homeomorphism \eqref{map:homeo} is equivariant with respect to these actions. Similarly, the map \eqref{map:embed} is a homomorphism in $G$-$\Bor$. 
We remark that if the action of $G$ on $V$ is metrically proper then so are the actions on $\fH$ and on $\R\times \fH$. In particular by \eqref{metprop}, we have
\begin{equation}\label{rmetprop}
\lim_{g\to\infty}||g(r,v)||=\infty\quad (r,v)\in \R_+\times \fH.
\end{equation}
The following lemma is an immediate consequence of \eqref{rmetprop}.

\begin{lem}\label{lem:nocorta}
Let $G$ act on $V$ by affine isometries. Assume that the action is metrically proper. Let $X,Y\subset \R_+\times \fH$ be bounded subsets and let $\cG\subset G$ be a finite subset. Then the set
\[
\tilde{\cG}_{X,Y}=\{h\in G:\cG X\cap hY\ne\emptyset\} 
\]
is finite.
\end{lem}

Let $\fX$ be as in \eqref{elfx}. In \eqref{elU} we have introduced the open subsets $U(S,r,v,\delta)\subset\fX$. We shall also consider the compact subsets
\begin{equation}\label{elW}
W(S,r,v,\delta)=p_S^{-1}(D_S((r,v),\delta)) \quad (S\in\cF(V), (r,v)\in \R_+\times S, \delta>0).
\end{equation}
Consider the stabilizer subgroup of an element $v\in V$:
\[
G_v=\left\{g \in G \colon gv=v\right\}.
\]
If the action of $G$ on $V$ is metrically proper, then $G_v$ is finite for all $v\in V$. 

\begin{lem}\label{lem:deltachico}
Let $\fX$ be as in \eqref{elfx} and let $(r,v)\in \R_+\times V$. Let $G$ act on $V$ by affine isometries. Assume that the action is metrically proper. Then there exist a precompact open neighborhood $(r,v)\in U\subset\fX$ and an affine subspace $S\in\cF(V)$ such that
\item[i)] $U=p_{S}^{-1}(U\cap (\R_+\times S))$. 
\item[ii)]
\[
gU\cap U=\left\{\begin{matrix}U& g\in G_v\\ \emptyset & g\notin G_v\end{matrix}\right.
\]
\end{lem}
\begin{proof}
Let $S_1\in\cF(V)$ such that $v\in S_1$. Because $G_v$ is finite, the affine subspace $S^\prime_1$ generated by the orbit $G_v S_1$ is in $\cF(V)$. Hence, upon replacing $S_1$ by $S^\prime_1$ if necessary, we may assume that 
\begin{equation}\label{gv-stable}
S_1=G_vS_1. 
\end{equation}
Let $\delta>0$ and let $W=W(S_1,r,v,\delta)$. By definition, an element $(x,h)\in \fX$ is in $W$ if and only if
\begin{equation}\label{hdelta}
\delta^2\ge (\sqrt{x^2+||\pi_{S_1}^{\perp}(h)||^2}-r)^2+||\pi_{S_1}(h)-v||^2.
\end{equation}
We may rewrite the right hand side of \eqref{hdelta} as 
\begin{equation}\label{ainfty}
x^2+||h-v||^2+r^2-2r\sqrt{x^2+||\pi_{S_1}^{\perp}(h)||^2}.
\end{equation}
Observe that if $g\in G_v$, then 
\[
||gh-v||=||gh-gv||=||h-v||.
\] 
Moreover for $\ell=\ell_g$ as in \eqref{ellytau}, using \eqref{gv-stable} in the second identity, we have
\[
||\pi_{S_1}^{\perp}(gh)||=||\ell(\pi_{g^{-1}S_1}^{\perp}(h))||=||\pi_{S_1}^{\perp}(h)||.
\]
We have shown that $G_vW=W$. Observe also that the expression \eqref{ainfty} goes to infinity as $||h||$ does. 
In particular the map $(x,h)\to ||h||$ is bounded on $W$. Hence by \eqref{metprop}
$W\cap G(r,v)$ is finite. Taking $\delta$ sufficiently small, we obtain $W\cap G(r,v)=\{(r,v)\}$. By Lemma \ref{lem:nocorta}, the set 
$\cG=\{g\in G:W\cap gW\ne\emptyset\}\setminus G_v$ is finite.  Let $U_1=U(S_1,r,v,\delta)$; put
\[
U=U_1\setminus(\cG W).
\]
Let $S=\cG S_1$. Then $U$ is precompact and satisfies both (i) and (ii).
\end{proof}

An open set $U\subset \fX$ is called \emph{$G$-admissible} if it admits a finite open covering 
\begin{equation}\label{g-admi}
U=\bigcup_{i=1}^nU_i
\end{equation}
such that each $U_i$ is precompact and satisfies the conditions of Lemma \ref{lem:deltachico} for some $(r_i,v_i)\in U_i$.

Let $U\subset \fX$ be an open subset. Assume that there exists $S\in\cF(V)$ such that $U=p_S^{-1}(U\cap(\R_+\times S))$. Then if $\cG\subset G$ is finite and $T\supset \cG S$, we have $\cG U=p_T^{-1}((\cG U)\cap(\R_+\times T))$. Hence the algebra 
$\cZc(\cG U)$ is defined by \eqref{acl}. Put
\begin{equation}\label{acgu}
\cZc(G,U)=\colim_{\cG\subset G}\cZc(\cG U).
\end{equation}
Here the colimit runs over the finite subsets $\cG\subset G$.
\begin{lem}\label{lem:admicoli}
Let $G$ be a discrete group acting on $V$ by affine isometries. Assume that the action is metrically proper. Then 
\[
\cZc(V)=\colim_{U}\cZc(G,U),
\] 
where the colimit runs over the $G$-admissible open subsets of $\fX$.
\end{lem}
\begin{proof}
Let $U_\rho=U(0,0,0,\rho)\subset\fX$. We have $U_\rho\cap (\R_+\times S)=\Da_S((0,0),\rho)$ for every $0\in S\in\cF(V)$. Thus
\[
\cZc(V)=\colim_{0\in S}\colim_\rho\cZc(\Da_S((0,0),\rho))=\colim_\rho\cZc(U_\rho).
\]
Because $U_\rho\subset W(0,0,0,\rho)$, which is compact, there is a $G$-admissible open subset $U_\rho\subset U\subset \fX$ $(\rho>0)$, by Lemma \ref{lem:deltachico}. On the other hand, since a $G$-admissible open set is precompact by definition, it is bounded, whence it is contained in some $U_{\rho}$. Hence
\[
\colim_\rho\cZc(U_\rho)=\colim_U\cZc(U)=\colim_U\cZc(G,U),
\]
where the last two colimits are taken over the $G$-admissible open sets $U\subset \fX$. This completes the proof. 
\end{proof}

Let $U\subset\fX$ be a $G$-admissible open subset and let $\sU=\{U_1,\dots,U_n\}$ and $v_1,\dots, v_n$ be as in \eqref{g-admi}. We may choose 
$S\in\cF(V)$ such that 
\begin{equation}\label{elsdelsu}
U_i=\pi_{S}^{-1}(U_i\cap (\R_+\times S)),\quad (i=1,\dots,n).
\end{equation} 
Write 
\begin{gather*}
G_i=G_{v_i},\quad U_{<i}=\bigcup_{j<i}U_j.
\end{gather*}
Let $\cG\subset G$ be a finite subset. With the notations of \eqref{acl} and of Lemma \ref{lem:nocorta}, put
\begin{gather*}
\tilde{\cG}^i=\tilde{\cG}_{U,U_{<i}},\\
\cZc^i(\cG,\sU)=\cZc(\cG U\setminus\tilde{\cG}^iU_{<i}).
\end{gather*}
Observe that
\begin{gather}\label{cgij}
i<j\Rightarrow \tilde{\cG}^i\subset\tilde{\cG}^j, \nonumber\\
\cG U\setminus\tilde{\cG}^iU_{<i}=\cG U\setminus\tilde{\cG}^jU_{<i}\supset\cG U\setminus\tilde{\cG}^jU_{<j}.
\end{gather}
Moreover, $\cG U\setminus\tilde{\cG}^jU_{<j}$ is closed in $\cG U\setminus\tilde{\cG}^iU_{<i}$  $(i<j)$. With the notation of \eqref{ilm}, put
\begin{equation}\label{jicgsu}
J^i(\cG,\sU)=I(\cG U\setminus\tilde{\cG}^iU_{<i},\cG U\setminus\tilde{\cG}^{i+1}U_{<i+1}).
\end{equation}
By Lemma \ref{lem:tietze} we have an exact sequence
\[
0\to J^i(\cG,\sU)\to \cZc^i(\cG,\sU)\to \cZc^{i+1}(\cG,\sU)\to 0.
\]
Note that 
\begin{equation}\label{topfil}
\cZc^{n+1}(\cG,\sU)=0,\quad J^n(\cG,\sU)=\cZc^n(\cG,\sU).
\end{equation}
If $\cH\supset\cG$ is another finite subset of $G$, then $\cG U\setminus (\tilde{\cG}^iU_{<i})$ is open in $\cH U\setminus (\tilde{\cH}^iU_{<i})$. Hence $\cZc^i(\cG,\sU)\subset\cZc^{i}(\cH,\sU)$ and thus the algebraic colimit
\begin{equation}\label{aicgsu}
\cZc^i(G,\sU)=\colim_{\cG}\cZc^i(\cG,\sU)
\end{equation}
is in $G$-$\Bor$. One checks that restriction maps induce an equivariant map $\cZc^i(G,\sU)\to\cZc^{i+1}(G,\sU)$, and so for 
\begin{equation}\label{jigsu}
J^i(G,\sU)=\colim_{\cG}J^i(\cG,\sU)
\end{equation}
we have an exact sequence in $G$-$\Bor$
\begin{equation}\label{seq:caci}
0\to J^i(G,\sU)\to \cZc^i(G,\sU)\to \cZc^{i+1}(G,\sU)\to 0.
\end{equation}

\begin{lem}\label{lem:calprop}
Let $J^i(G,\sU)\in G$-$\Bor$ be as in \eqref{jigsu} ($1\le i\le n$). Then $J^i(G,\sU)$ is proper over $G/G_{i}$. 
\end{lem}
\begin{proof}
Let $S\in\cF(V)$ be as in \eqref{elsdelsu}, and let $\cG\subset G$ be a finite subset. By \eqref{ilm} and \eqref{jicgsu}, $J^i(\cG,\sU)$ is the colimit, over 
$T\supset \tilde{\cG}^iS$, of the ideals $J^i(\cG,\sU,T)\triqui\cZc(T\cap (\cG U\setminus\tilde{\cG}^iU_{<i}))$ of those functions $f$ which vanish outside $\tilde{\cG}^{i+1} U_i$. Let $\ol{\cG^{i+1}}$ be the image of $\tilde{\cG}^{i+1}$ in $G/G_i$. By our hypothesis on 
$U_i$, $\tilde{\cG}^{i+1} U_i$ is the disjoint union of the open subsets $\bar{g}U_i$ ($\bar{g}\in\ol{\cG^{i+1}}$). Let $J^i(\cG,\sU,T)_{\bar g}\subset J^i(\cG,\sU,T)$ be the subalgebra of those functions $f$ which vanish outside $\bar{g}U_i$. Then 
\begin{gather*}
J^i(\cG,\sU,T)=\bigoplus_{\bar{g}\in\ol{\cG^{i+1}}}J^i(\cG,\sU,T)_{\bar g}\\
\mbox{ and } J^i(\cG,\sU,T)_{\bar g}J^i(\cG,\sU,T)_{\bar h}=0 \mbox{ if } \bar{g}\neq\bar{h}.
\end{gather*}
Hence $J^i(\cG,\sU,T)$ is an algebra over 
$\C^{(\ol{\cG^{i+1}})}$ such that $\C^{(\ol{\cG^{i+1}})}J^i(\cG,\sU,T)=J^i(\cG,\sU,T)$. One checks that this structure is compatible with the maps
$J^i(\cG,\sU,T)\to J^i(\cG,\sU,T')$, and so we get a $\C^{(\ol{\cG^{i+1}})}$-algebra structure on $J^i(\cG,\sU)$ with $\C^{(\ol{\cG^{i+1}})}J^i(\cG,\sU)=J^i(\cG,\sU)$. Passing to the colimit along the inclusions $\cG\subset\cH$ one obtains a full compatible $(G,\C^{(G/G_i)})$-algebra structure on $J^i(G,\sU)$.
\end{proof}

\noindent{\emph{Proof of Theorem} \ref{thm:proper}.}
Because $\cZc(V)\subset\cAc(V)$ is a central $G$-subalgebra, $\cAc(V)$ carries a canonical compatible $(G,\cZc(V))$-structure. Moreover, by \eqref{unicond} we have
\begin{equation}
\cAc(V)=\cZc(V)\cAc(V).
\end{equation}
Condition \eqref{condibornofull} also holds because it holds for the action of $\cZc(S)$ on $\cAc(S)$ ($S\in\cF(V)$). 
Let $U\subset\fX$ be a $G$-admissible open subset. Put
\[
\cAc(G,U)=\cZc(G,U)\cAc(V).
\]
Because we are assuming that $\E$ and $\F$ commute with filtering colimits up to homotopy, it suffices, in view of Lemma \ref{lem:admicoli}, to prove that $\tau(\cAc(G,U))$ is a weak equivalence for every $G$-admissible open subset 
$U\subset\fX$. Let $\sU=\{U_1,\dots,U_n\}$ be as in \eqref{g-admi}. Define inductively
\begin{gather*}
\cAc^1(G,\sU)=\cAc(G,U),\quad I^i(G,\sU)=J^i(G,\sU)\cAc^i(G,\sU),\\ \cAc^{i+1}(G,\sU)=\cAc^i(G,\sU)/I^i(G,\sU).
\end{gather*}
By \eqref{topfil}, we have
\[
\cAc^{n+1}(G,\sU)=0,\quad I^n(G,\sU)=\cAc^n(G,\sU).
\]
Hence because we are assuming that $\E$ and $\F$ satisfy excision, by \eqref{seq:caci} and induction, we can further reduce to proving that $\tau(I^i(G,\sU))$ is a weak equivalence $(1\le i\le n)$. This follows from Lemma \ref{lem:iprop-jprop}, Lemma \ref{lem:calprop}, and the hypothesis that $\tau(P)$ is a weak equivalence whenever $P$ is proper over $G/H$ and $H$ is finite.\qed

\section{Main results}\label{sec:main}

Let $G$ be a group and $\fin$ the family of its finite subgroups. An equivariant map $f:X\to Y$ of $G$-spaces is called a \emph{$\fin$-equivalence} if $f:X^H\to Y^H$ is a weak equivalence for $H\in\fin$.

\begin{thm}\label{thm:main}
Let $G$ be a countable discrete group. Let $B\in G$-$\Bor$, let $I$ be a $K$-excisive $G$-ring,  let $\sotimes$ be the spatial tensor product, and let $\cK=\cK(\ell^2(\N))$ be the algebra of compact operators; equip $\cK$ with the trivial $G$-action. Assume that $G$ acts metrically properly by affine isometries on a countably infinite dimensional Euclidean space $V$. Then the functor $H^G(-,K(I\otimes(B\sotimes\cK)))$ sends $\fin$-equivalences of $G$-spaces to weak equivalences of spectra.
\end{thm}
\begin{proof}
Let $Z$ be
a $G$-space; consider the functor $\E_Z:G$-$\Bor\to\spt$,
\[
\E_Z(A)=H^G(Z,K(I\otimes(A\motimes\cK\motimes B))).
\]
We must prove that if $X\to Y$ is a $\fin$-equivalence then $\E_X(\C)\to\E_Y(\C)$ is a weak equivalence. By Corollary \ref{cor:bott} it suffices to show that 
$\E_X(\cAc(V))\to \E_Y(\cAc(V))$ is a weak equivalence. By Theorem \ref{thm:kstab}, the functor $\E_Z$ is excisive, homotopy invariant and $G$-stable. Moreover, it commutes with filtering colimits along injective maps up to weak equivalence, since algebraic $K$-theory commutes with arbitrary filtering algebraic colimits up to weak equivalence. Therefore, by Theorem \ref{thm:proper} we are reduced to proving that if $H\subset G$ is a finite subgroup and $A\in G$-$\Bor$ is proper over $G/H$, then 
\begin{equation}\label{map:gindweq}
\E_X(A)\to \E_Y(A)
\end{equation}
is a weak equivalence. By Lemma \ref{lem:propertenso}, $C=A\motimes\cK\motimes B$ is proper over $G/H$ as a $*$-algebra, and thus $I\otimes C$ is proper over $G/H$ as a ring.
This finishes the proof, since we know from \cite{corel}*{Proposition 4.3.1, Lemma 11.1, and Theorem 11.6}, that if $H$ is finite and $J$ is a $K$-excisive $G$-ring, proper over $G/H$, then $H^G(-,K(J))$ maps $\fin$-equivalences to weak equivalences.
\end{proof}

\begin{cor}\label{cor:main1} (Farrell-Jones' conjecture)
Let $G$, $I$, $B$ and $\cK$ be as in Theorem \ref{thm:main}. Then the assembly map
\[
H^G(\cE(G,\vcyc),K(I\otimes (B\sotimes\cK)))\to K((I\otimes (B\sotimes\cK))\rtimes G)
\]
is a weak equivalence.
\end{cor}
\begin{proof}
The assembly map is induced by $\cE(G,\vcyc)\to pt$, which is a $\vcyc$-equivalence, and therefore a $\fin$-equivalence.
\end{proof}

If $\fB$ is a $C^*$-algebra then by Suslin-Wodzicki's theorem (Karoubi's conjecture) \cite{sw}*{Theorem 10.9} and stability of $K^{\top},$ we have a weak equivalence
\[
K(\fB\sotimes\cK)\weq K^{\top}(\fB\sotimes\cK)\overset{\sim}\longleftarrow K^{\top}(\fB).
\] 
If $G$ is a group and $\fA$ is a $G$-$C^*$-algebra then
\[
(\fA\sotimes\cK)\rtimes G\subset C_\red^*(G,\fA\sotimes \cK)\cong C_\red^*(G,\fA)\sotimes\cK.
\]
Thus there is a map
\begin{equation}\label{map:compaG}
K((\fA\sotimes\cK)\rtimes G)\to K^{\top}(C^*_\red(G,\fA)).
\end{equation}
\begin{cor}\label{cor:main2}
Let $G$ be as in Theorem \ref{thm:main} and let $\fA$ be a separable $G$-$C^*$-algebra. Then \eqref{map:compaG} is a weak equivalence. 
\end{cor}
\begin{proof}
We have a homotopy commutative diagram
\begin{equation}\label{diag:bcfj}
\xymatrix{H^G(\cE(G,\fin),K(\fA\sotimes\cK))\ar[d]\ar[r]&K((\fA\sotimes\cK)\rtimes G)\ar[d]\\
          H^G(\cE(G,\fin),K^{\top}(\fA))\ar[r]& K^{\top}(C_\red^*(G,\fA)).}
\end{equation}
By Corollary \ref{cor:main1} the top horizontal arrow in \eqref{diag:bcfj} is a weak equivalence. By \cite{hp}*{Corollary 8.4}, the bottom arrow is equivalent to the Baum-Connes assembly map, which is an equivalence for Haagerup groups, by \cite{hk2}*{Theorem 1.1}. It follows from the Suslin-Wodzicki theorem \cite{sw}*{Theorem 10.9} that the map \eqref{map:compaG} is an equivalence for finite $G$. Since $\cE(G,\fin)$ has finite stabilizers,
the latter fact implies that the vertical map on the left is a weak equivalence. This concludes the proof.
\end{proof}

\end{document}